\documentclass[11pt,oneside]{amsart}
\usepackage{amssymb, amsmath, amsthm}

\usepackage{epsfig}
\usepackage{graphicx}
\usepackage{color}
\usepackage{mathptmx}
\usepackage{enumerate}
\usepackage[colorlinks=true, linkcolor=red, citecolor=red]{hyperref}
\usepackage{amsrefs}
\usepackage{fullpage}
\usepackage{pinlabel}

\theoremstyle{plain}
\newtheorem{theorem}{Theorem}[section]
\newtheorem*{MainTheorem}{Theorem \ref{Main Theorem} (rephrased)}

\newtheorem*{theorem*}{Theorem}

\newtheorem{corollary}[theorem]{Corollary}
\newtheorem{lemma}[theorem]{Lemma}

\theoremstyle{definition}
\newtheorem{remark}[theorem]{Remark}

\theoremstyle{definition}

\newcommand{\R}{{\mathbb R}}

\newcommand{\Z}{\mathbb Z}
\newcommand{\N}{\mathbb N}
\newcommand{\Q}{\mathbb Q}

\newcommand{\defn}[1]{\emph{#1}}
\newcommand{\boundary}{\partial}
\newcommand{\mc}[1]{\mathcal{#1}}

\newcommand{\inter}[1]{\mathring{#1}} %interior
\newcommand{\co}{\mskip0.5mu\colon\thinspace}

\newcommand{\spacing}{ \parskip 6.6pt \parindent 0pt}
 
 \spacing 
\begin{document}

   % title

   \title[]{New examples of Brunnian theta graphs}
   \author{Byoungwook Jang}
   \author{Anna Kronaeur}
   \author{Pratap Luitel}
   \author{Daniel Medici}
   \author{Scott A. Taylor}
   \author{Alexander Zupan}
   % Note that the short title for running heads goes in square
   % brackets.  This is optional.  The long title goes in curly
   % braces.  In the long title, line breaks are indicated by \\.

  \begin{abstract}
The Kinoshita graph is the most famous example of a Brunnian theta graph, a nontrivial spatial theta graph with the property that removing any edge yields an unknot.  We produce a new family of diagrams of spatial  theta graphs with the property that removing any edge results in the unknot. The family is parameterized by a certain subgroup of the pure braid group on four strands. We prove that infinitely many of these diagrams give rise to distinct Brunnian theta graphs.
    \end{abstract}

   % today's date, or fill in whatever date you prefer
   \date{\today}

% This ends the top matter information.
% We can now tell LaTeX to display the top matter.

   \maketitle
% The body  

\section{Introduction}

A \defn{spatial theta graph} is a theta graph (two vertices and three edges, each joining the two vertices) embedded in the 3-sphere $S^3$. There is a rich theory of spatial theta graphs and they show up naturally in knot theory. (For instance, the union of a tunnel number 1 knot with a tunnel having distinct endpoints is a spatial theta graph.) A \defn{trivial spatial theta graph} is any spatial theta graph which is isotopic into a 2-sphere in $S^3$. A spatial theta graph $G \subset S^3$ has the \defn{Brunnian property} if for each edge $e \subset G$ the knot $K_e = G\setminus e$ which is the result of removing the interior of $e$ from $G$ is the unknot. A spatial theta graph  is \defn{Brunnian} (or \emph{almost unknotted} or \emph{minimally knotted}) if it is non-trivial and has the Brunnian property.

By far the best known Brunnian theta graph is the Kinoshita graph \cite{Kinoshita1, Kinoshita2}. The Kinoshita graph was generalized by Wolcott \cite{Wolcott} to a family of Brunnian theta graphs now called the Kinoshita-Wolcott graphs. They are pictured in Figure \ref{Kinoshita-Wolcott}. Inspection shows that they have the Brunnian property. There are several approaches to showing that the Kinoshita graph (and perhaps all of the Kinoshita-Wolcott graphs) are non-trivial: Wolcott \cite{Wolcott} uses double branched covers; Litherland \cite{Litherland} uses a version of the Alexander polynomial; Scharlemann \cite{Scharlemann} and Livingston \cite{Livingston} use representations of certain associated groups; McAtee, Silver, and Williams \cite{MSW} use quandles; Thurston \cite{Thurston} showed that the Kinoshita graph is hyperbolic (i.e. the exterior supports a complete hyperbolic structure with totally geodesic boundary.) 

\begin{figure}[ht]
\labellist \small\hair 2pt
\pinlabel $-i$ at 223 433
\pinlabel $-j$ at 83 137
\pinlabel $-k$ at 380 140
\endlabellist

\includegraphics[scale=.5]{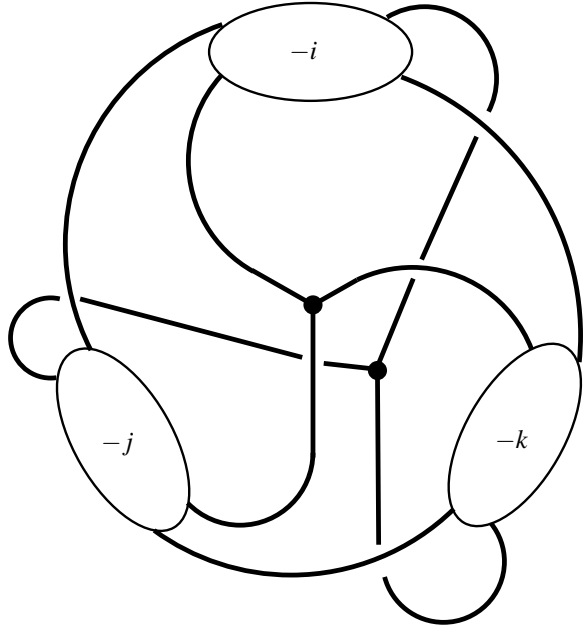}
\caption{The Kinoshita-Wolcott graphs (figure based on \cite[Figure 4]{Litherland}). The labels $-i$, $-j$, and $-k$ indicate the number of full twists in each box (with the sign of $-i$, $-j$, $-k$ indicating the direction of the twisting.) If $i=j=k=1$, the graph is the Kinoshita graph. If one of $i,j,k$ is zero, then the graph is trivial; otherwise, it is Brunnian \cite[Theorem 2.1]{Wolcott}.}
\label{Kinoshita-Wolcott}
\end{figure}

In this paper, we produce an infinite family of diagrams for spatial theta graphs $G(A,t_1,t_2)$ having the Brunnian property. These graphs depend on braids $A$ lying in a certain subgroup of the pure braid group on 4 strands and on integers $t_1, t_2$ which represent certain twisting parameters. Our main theorem shows that infinitely many braids $A$ give rise to Brunnian theta graphs.

\begin{MainTheorem}
For all $n \in \Z$, there exists a braid $A_n$ such that for all $m \in \Z$, the graph $\Gamma(n,m) = G(A_n, -n, m)$ is a Brunnian theta graph. Furthermore, suppose that for a given $(n,m) \in \Z \times \Z$, the set $S(n,m) \subset \Z\times \Z$ has the property that if $(a,b) \in S(n,m)$ then $\Gamma(a,b)$ is isotopic to $\Gamma(n,m)$ and if $(a,b), (a',b') \in S(n,m)$ are distinct, then $a + b \neq a' + b'$. Then $S(n,m)$ has at most three distinct elements. In particular, there exist infinitely many $n \in \Z$ such that the graphs $\Gamma(n,0)$ are pairwise non-isotopic Brunnian theta graphs.
\end{MainTheorem}

\subsection{Acknowledgements}
We thank the attendees at the 2013 Spatial Graphs conference for helpful discussions, particularly Erica Flapan and Danielle O'Donnol. We are also grateful to Ryan Blair,  Ilya Kofman, Jessica Purcell, and Maggy Tomova for helpful conversations. This research was partially funded by Colby College. 

\section{Notation}
We work in either the PL or smooth category. For a topological space $X$, we let $|X|$ denote the number of components of $X$. If $Y \subset X$ then $\eta(Y)$ is a closed regular neighborhood of $Y$ in $X$ and $\inter{\eta}(Y)$ is an open regular neighborhood. More generally, $\inter{Y}$ denotes the interior of $Y$. 

\section{Constructing new Brunnian theta graphs}

There are two natural methods for constructing new Brunnian theta graphs: vertex sums and clasping.

\subsection{Vertex sums}

Suppose that $G_1 \subset S^3 $ and $G_2 \subset S^3$ are spatial theta graphs. Let $v_1 \in G_1$ and $v_2 \in G_2$ be vertices. We can construct a new spatial theta graph $G_1 \#_3 G_2 \subset S^3$ by taking the connected sum of $S^3$ with $S^3$ by removing regular open neighborhoods of $v_1$ and $v_2$ and gluing the resulting 3-balls $B_1$ and $B_2$ together by a homeomorphism $\boundary B_1 \to \boundary B_2$ taking the punctures $G_1 \cap \boundary B_1$ to the punctures $G_2 \cap \boundary B_2$. See Figure \ref{Vertex Sum}. The subscripted 3 represents the fact that we are performing the connected sum along a trivalent vertex and is used to distinguish the vertex sum from the connected sum of graphs occuring along edges of a graph (which, when both $G_1$ and $G_2$ are theta graphs, does not produce a theta graph.)

An orientation on a spatial theta graph is a choice of one vertex to be the \defn{source}, one vertex to be the \defn{sink}, and a choice of a total order on the edges of the graph. If $G_1$ and $G_2$ are oriented theta graphs, we insist that the connected sum produce an oriented theta graph (so that the sink vertex of $G_1$ is glued to the source vertex of $G_2$ and so that the edges of $G_1 \#_3 G_2$ can be given an ordering which restricts to the given orderings on the edges of $G_1$ and $G_2$.  Wolcott \cite{Wolcott} showed that  vertex sum of oriented theta graphs is independent (up to ambient isotopy of the graph) of the choice of homeomorphism $\boundary B_1 \to \boundary B_2$. 

\begin{figure}[ht]
\includegraphics[scale=.5]{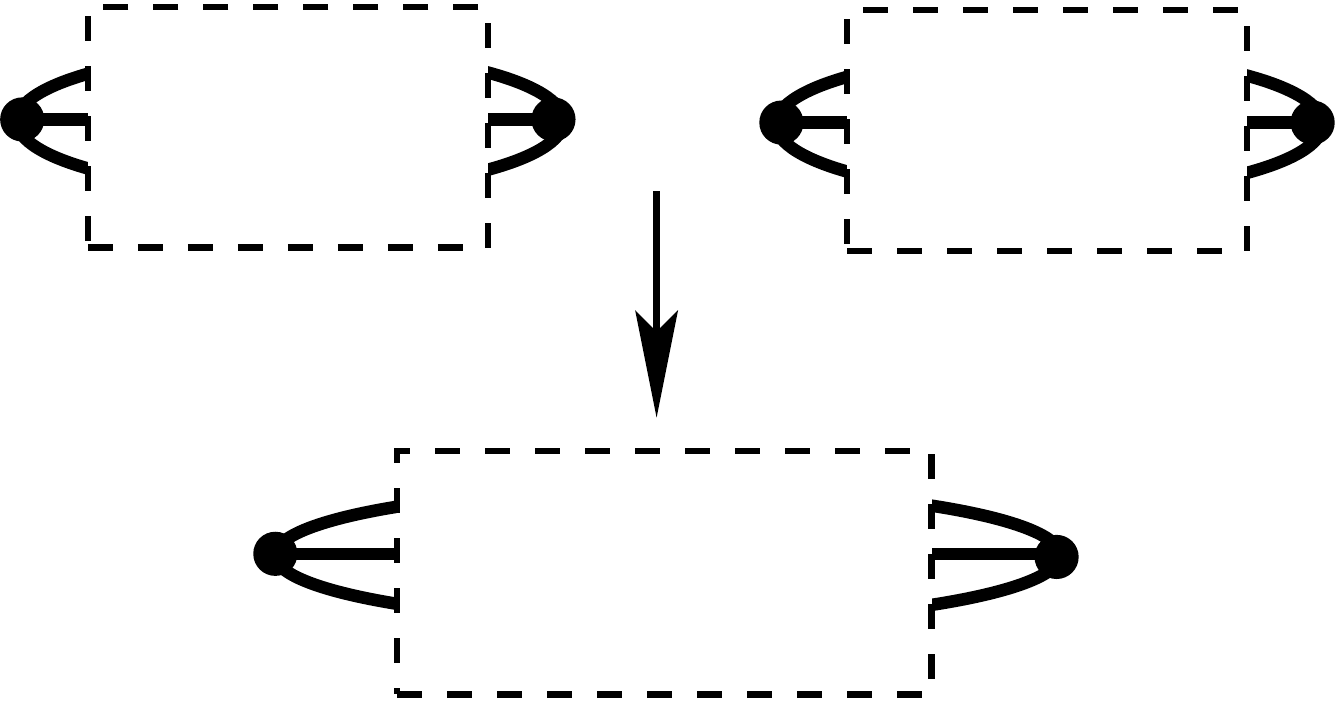}
\caption{A schematic depiction of the vertex sum of two spatial theta graphs}
\label{Vertex Sum}
\end{figure}

If $G_1$ and $G_2$ both have the Brunnian property, then $G_1\#_3 G_2$ does as well since the connected sum of two knots is the unknot if and only if both of the original knots are unknots. If $G_1$ (say) is trivial, then $G_1 \#_3 G_2$ is isotopic to $G_2$. Similarly, if at least one of $G_1$ or $G_2$ is non-trivial then $G_1 \#_3 G_2$ is non-trivial \cite{Wolcott}. Consequently:

\begin{theorem}[Wolcott]
If $G_1$ and $G_2$ are Brunnian theta graphs, then $G_1 \#_3 G_2$ is a Brunnian theta graph.
\end{theorem}

We say that a spatial theta graph is \defn{vertex-prime} if it is not the vertex sum of two other non-trivial spatial theta graphs. The Kinoshita graph is vertex prime \cite{Calcut}. Using Thurston's hyperbolization theorem for Haken manifolds, it is possible to show that if $G_1$ and $G_2$ are theta graphs, then $G_1 \#_3 G_2$ is hyperbolic if and only if $G_1$ and $G_2$ are hyperbolic.

\subsection{Clasping}

Clasping \cite{SW} is a second method for converting a Brunnian theta graph into another theta graph with the Brunnian property. To explain it, suppose that $G$ is a spatial theta graph in $S^3$ which has been isotoped so that its intersection with a 3-ball $B \subset S^3$ consists of four unknotted arcs (as on the left of Figure \ref{Clasp Move}), numbered $\alpha_1,\alpha_2, \alpha_3, \alpha_4, \alpha_5$. Assume that the first arc and the last two arcs belong to the same edge (the ``red edge'') of $G$ and that the others belong to different, distinct edges of $G$ (the ``green edge'' and the ``blue edge''). We require that as we traverse the red edge, the arc $\alpha_1$ is traversed between $\alpha_4$ and $\alpha_5$. Letting $e'$ be the sub-arc of the red edge containing $\alpha_4 \cup \alpha_1 \cup \alpha_5$, we also require that there is an isotopy of $e'$, in the complement of the rest of the graph, to an unknotted arc in $B$. As in Figure \ref{Clasp Move}, we may then perform crossing changes to introduce a clasps between adjacent arcs. It is easily checked that this clasp move preserves the Brunnian property.  

\begin{center}
\begin{figure}[ht]
\centering
\labellist \small\hair 2pt
\pinlabel $\alpha_1$ [r] at 38 131
\pinlabel $\alpha_2$ [r] at 97 137
\pinlabel $\alpha_3$ [l] at 137 137 
\pinlabel $\alpha_4$ [tr] at 191 177
\pinlabel $\alpha_5$ [br] at 195 96
\endlabellist
\includegraphics[scale=.5]{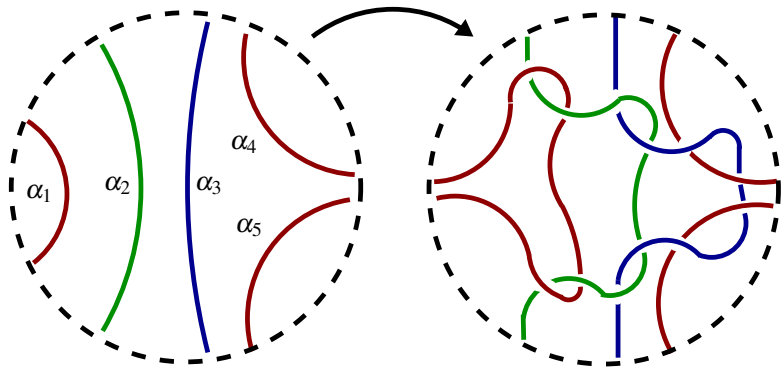}
\caption{The clasp move in the case when all three edges of the graph are involved.}
\label{Clasp Move}
\end{figure}
\end{center}

Although the clasp move creates many Brunnian theta graphs, it is not clear how to keep track of fundamental properties (such as hyperbolicity) under the clasp move. Additionally, very little is known about sequences of clasp moves relating two Brunnian theta graphs. 

We can, however, use clasping to show that there exist Brunnian theta graphs which are \emph{not} hyperbolic. Figure \ref{ToroidalBrunnian} shows an example of a Brunnian theta graph with an essential torus in its exterior. It was created by isotoping a Kinoshita-Wolcott graph to the position required to apply the clasping move via an isotopy which moved a point on one of the edges around a trefoil knot. A graph with an essential torus in its exterior is neither hyperbolic nor a trivial graph.

\begin{figure}[ht]
\includegraphics[scale=.5]{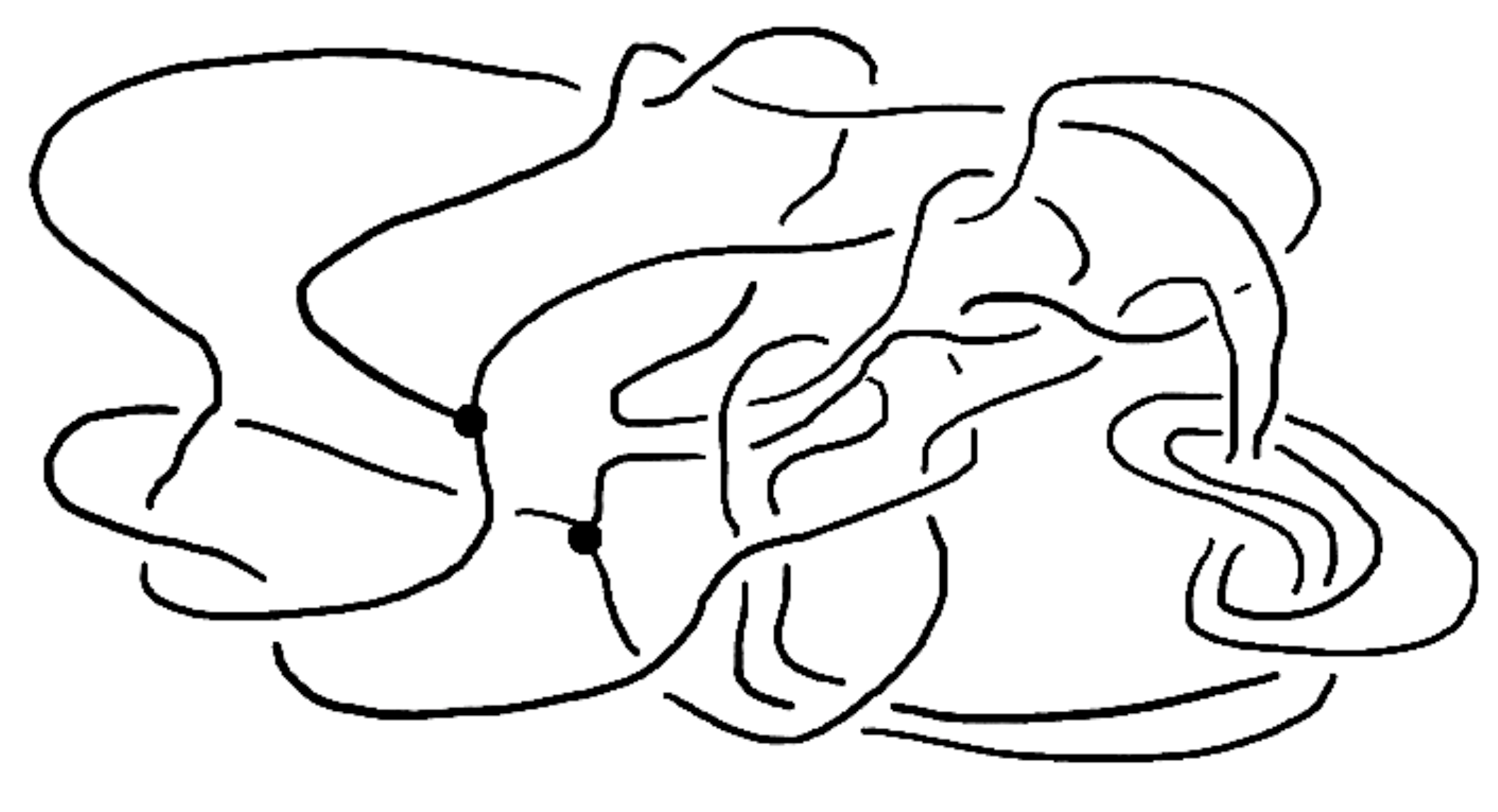}
\caption{A toroidal Brunnian theta graph. The swallow-follow torus for the double-stranded trefoil on the right is an essential torus in the exterior of the theta graph. }
\label{ToroidalBrunnian}
\end{figure}

\section{New Examples of Brunnian Theta Graphs}

Besides the Kinoshita graph and vertex sums of the Kinoshita graph with itself, are there other hyperbolic Brunnian theta graphs? In this section, we give a new infinite family of examples of diagrams of spatial theta curves. In the next section we will prove that infinitely many of them are also non-trivial. These examples have the property that they are of ``low bridge number''. Forthcoming work \cite{TT} will show that this implies that these graphs are vertex-prime. Furthermore, since they are low bridge number it is likely that they are hyperbolic. Section \ref{Questions} concludes this paper with some questions for further research. 

A pure $n$-braid representative consists of $n$ arcs (called strands)  in $Q = \{(x,y,z) \in \R^3 : -1 \leq z\leq 1\}$ such that the $i$th strand has endpoints at $(i, 0, \pm 1)$ and for each arc projecting onto the $z$-axis is a strictly monontonic function. Two pure $n$-braid representatives are equivalent if there is an isotopy in $Q$ from one to the other which fixes $\boundary Q$. The set of equivalence classes is $PB(n)$. Two pure $n$-braid representatives can be ``stacked'' to create another pure $n$-braid representative by placing one on top of the other and then scaling in the $z$-direction by $1/2$. Applying this operation to equivalence classes we obtain the group operation for $PB(n)$.  If $\sigma$ and $\rho$ are elements of $PB(n)$, we let $\sigma \rho$ denote the braid having a representative created by stacking a representative for $\sigma$ on top of a representative for $\rho$ and then scaling in the $z$-direction by $1/2$.

Let $\phi\co PB(4) \to PB(2)$ be the homomorphism which forgets the last two strands. For each $A \in \ker \phi$ we will construct a family $G(A, t_1, t_2)$ for $t_1, t_2 \in \Z$ of theta graphs with the Brunnian property. We will construct $G(A,t_1,t_2)$ by placing braids into the boxes in the template shown in Figure \ref{Fig:Template}. Let $\rho \co PB(4) \to PB(6)$ be a monomorphism which ``doubles'' each of the last two strands of $A \in PB(4)$ (i.e. in $\rho(A)$ the 4th strand is parallel to the 3rd and the 6th strand is parallel to the 5th.) For a given $A \in \ker \phi$, we place $\rho(A)$ into the top braid box of Figure \ref{Fig:Template}. The shading indicates the doubled strands. There is more than one choice for the monomorphism $\rho$, as the doubled strands may be allowed to twist around each other (i.e. we may vary the framing). We will always choose the homomorphism determined by the ``blackboard framing'' (i.e. in our diagram the doubled strands are two edges of a rectangle embedded in the plane.) Into the second and fourth boxes from the top we place the braid $A^{-1}$. In the third box we place the element from $PB(2)$ consisting of two strands with $t_1$ full twists. We use the convention that, giving the strands a downward orientation, if $t_1 > 0$ there are $2|t_1|$ left-handed crossings and if $t_1 < 0$ there are $2|t_1|$ right-handed twists. Into the bottom box we place $t_2$ full twists, using the same orientation convention as for $t_1$. 
 
\begin{center}
\begin{figure}[h!]
\labellist \small\hair 2pt
\pinlabel $\rho(A)$ at 216 651
\pinlabel $A^{-1}$ at 161 504
\pinlabel $A^{-1}$ at 161 199
\pinlabel $t_1$ at 161 385
\pinlabel $t_2$ at 161 79
\endlabellist
\includegraphics[scale = .4]{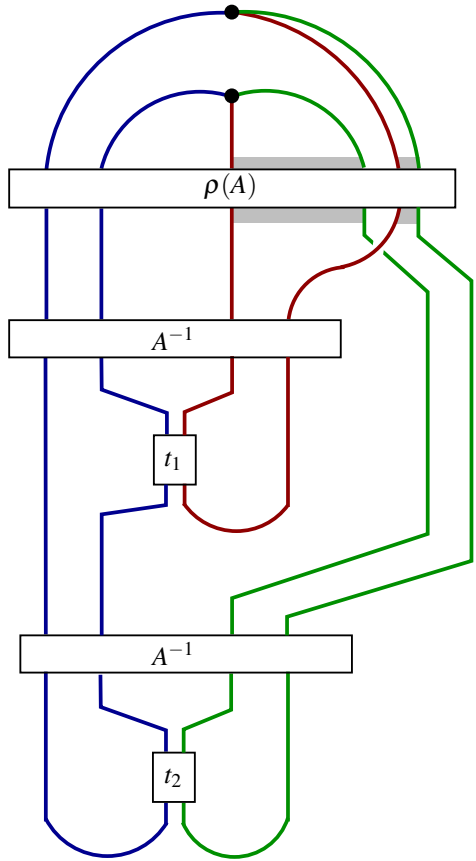}
\caption{The template for the graph $G_A$}
\label{Fig:Template}
\end{figure}
\end{center}

Considering the plane of projection in Figure \ref{Fig:Template} as the $xy$ plane, the plane $\Pi$ perpendicular to the plane of projection and cutting between the second and third boxes from the top functions as a ``bridge plane'' for $G(A, t_1,t_2)$. Observe that if we measure the height of a point $x \in G(A, t_1,t_2)$ by its projection onto the $y$-axis, each edge of $G(A, t_1, t_2)$ has a single local minimum for the height function and no other critical points in its interior. This implies that $\Pi$ cuts $G(A, t_1,t_2)$ into trees with special properties. The two trees above $\Pi$ have a single vertex which is not a leaf and their union is  isotopic (relative to endpoints) into $\Pi$.  The three trees below $\Pi$ are all edges (i.e. each is a tree with two vertices and single edge) and their union can be isotoped relative to the endpoints into $\Pi$. Thus, $\Pi$ is a bridge plane for $G(A, t_1,t_2)$ and $|G(A, t_1, t_2) \cap \Pi| = 6$. We might, therefore, say that $G(A, t_1,t_2)$ has ``bridge number at most 3''. The precise definition of bridge number for theta graphs has been a matter of some dispute (see, for example, \cite{Motohashi}). The forthcoming paper \cite{TT} explores bridge number for spatial graphs in detail.

\begin{theorem}\label{Thm: has Brunnian property}
For each $A \in \ker \phi$ and $t_1, t_2 \in \Z$, the graph $G(A, t_1,t_2)$ has the Brunnian property.
\end{theorem}

\begin{proof}
The proof is easy and diagrammatic. Color the edges coming out of the top vertex in the diagram in Figure \ref{Fig:Template} by (b)lue, (r)ed, and (v)erdant from left to right. Then the edges entering into the bottom vertex are also blue, red, and verdant from left to right. In Figure \ref{Fig:VerifyingBrunnian}, we have the knots $K_B$, $K_R$, and $K_V$ obtained by removing the blue, red, and verdant edges respectively. Observe that the top braid box of $K_R$ and $K_V$ contains the braid $A$.  In each of the diagrams for $K_B$, $K_R$, and $K_V$ we have labelled certain portions with lower case letters. We now explain those regions and why each diagram can be simplified to the standard diagram for the unknot.

Consider the diagram for $K_B$. Since the third and fourth strands of the top braid box of $G(A,t_1,t_2)$ are parallel, we may untwist the diagram at region $(a)$ and at region $(b)$. At regions $(c)$ and $(d)$, we may also untwist at the minima. The end result is a diagram of a knot having a single crossing. The knot $K_B$ must, therefore, be the unknot.

Consider the diagram for $K_R$. At region $(a)$ we have the trivial 2-braid since $A \in \ker \phi$. The braid $A$ in the top braid box may therefore be cancelled with the braid $A^{-1}$ in the third-from-the-top braid box. Finally, we may untwist the $t_2$ full twists in the final braid box to arrive at the standard diagram for the unknot.

Consider the diagram for $K_V$. The braids $A$ and $A^{-1}$ cancel, at which point we may untwist the $t_1$ full twists. We may also untwist at region (a). Thus, $K_V$ is also the unknot. 
\end{proof}

\begin{center}
\begin{figure}[ht]
\labellist \small\hair 2pt
\pinlabel $K_B$ [b] at 118 448
\pinlabel $(a)$ [r] at 110 383
\pinlabel $(b)$ [l] at 215 383
\pinlabel $(c)$ [l] at 150 251
\pinlabel $(d)$ [l] at 150 88
\pinlabel $K_R$ [b] at 405 448
\pinlabel $A$ at 405 350
\pinlabel $(a)$ at 321 272
\pinlabel $A^{-1}$ at 374 109
\pinlabel $t_2$ at 374 45
\pinlabel $K_V$ [b] at 685 448
\pinlabel $A$ at 684 353
\pinlabel $A^{-1}$ at 649 275
\pinlabel $t_1$ at 654 212
\pinlabel $(a)$ at 602 112
\endlabellist
\includegraphics[scale = .5]{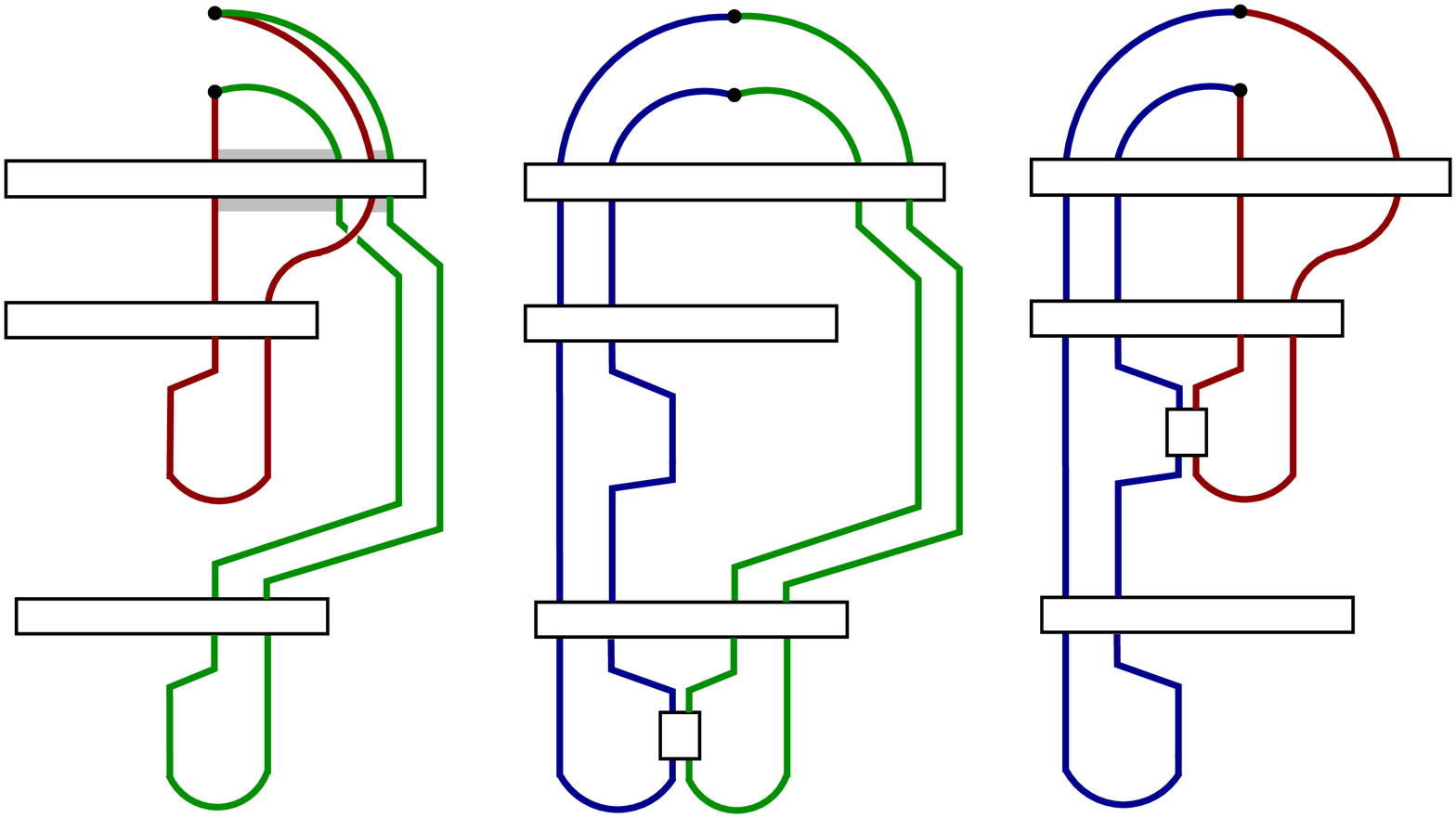}
\caption{The constituent knots $K_B$, $K_R$, and $K_V$ of $G(A,t_1,t_2)$.}
\label{Fig:VerifyingBrunnian}
\end{figure}
\end{center}

Given a graph $G(A, t_1,t_2)$ we can construct other theta graphs of bridge number at most 3 with the Brunnian property by using the clasping technique in such a way that we do not introduce any additional critical points in the interior of any edge, so it is highly unlikely that the template in Figure \ref{Fig:Template} encompasses all possible theta graphs of bridge number at most 3 with the Brunnian property. On the other hand, there are infinitely many braids $A$ such that $G(A,0,0)$ is a diagram of the trivial theta graph (see below), so the question as to what braids in $\ker \phi$ produce non-trivial theta graphs is somewhat subtle.

\section{Braids producing Brunnian theta graphs}

In this section we produce an infinite family of braids $A \in \ker \phi \subset PB(4)$ such that there exists $t_1$ such that for all $t_2$, $G(A,t_1,t_2)$ is Brunnian. To describe the braids $A$ more precisely, we recall the standard generating set for $PB(4)$.  For $i,j \in \{1,2,3,4\}$ with $i < j$, let $P_{ij}$ denote the element of $PB(4)$ obtained by ``looping'' the $i$th strand around the $j$th strand, as in Figure \ref{Fig:PB4generators}. Observe that $P_{23}^k$ produces a twist box in the 2nd and 3rd strands with $k$ full twists, using the sign convention from earlier.

\begin{center}
\begin{figure}[ht]
\labellist \small\hair 2pt
\pinlabel $P_{12}$ [b] at 73 410
\pinlabel $P_{13}$ [b] at 238 410
\pinlabel $P_{14}$ [b] at 417 410
\pinlabel $P_{23}$ [b] at 73 184
\pinlabel $P_{24}$ [b] at 248 184
\pinlabel $P_{34}$ [b] at 417 184
\endlabellist
\includegraphics[scale = .5]{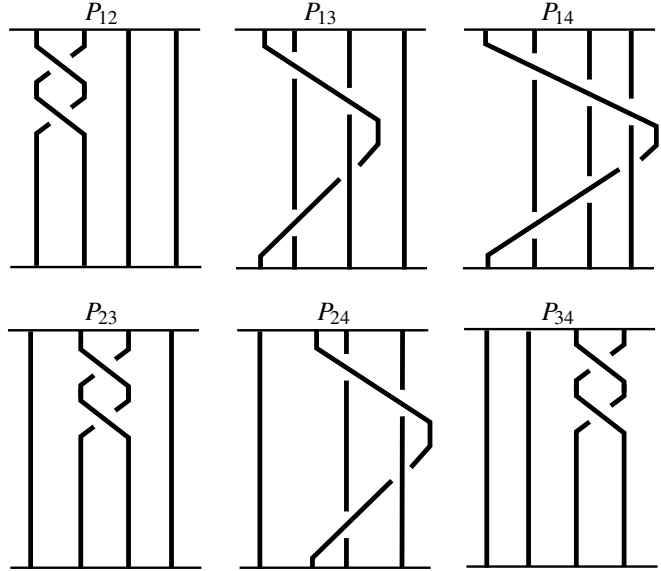}
\caption{The generators for $PB(4)$}
\label{Fig:PB4generators}
\end{figure}
\end{center}

There are non-trivial braids $A$ for which $G(A,0,0)$ is trivial. For example, for every $n$, $t_1$, and $t_2$, the graphs $G(P_{23}^n, t_1,t_2)$ are all trivial. To show that there are infinitely many braids producing non-trivial graphs, let $A_n = P_{23}^{-n}P_{13}$ and for $m \in \Z$, let $\Gamma(n,m) = G(A_n,-n,m)$ (see Figure \ref{Fig:SpecialBraid} for a diagram of $A_2$.)

\begin{center}
\begin{figure}[ht]
\includegraphics[scale = .7]{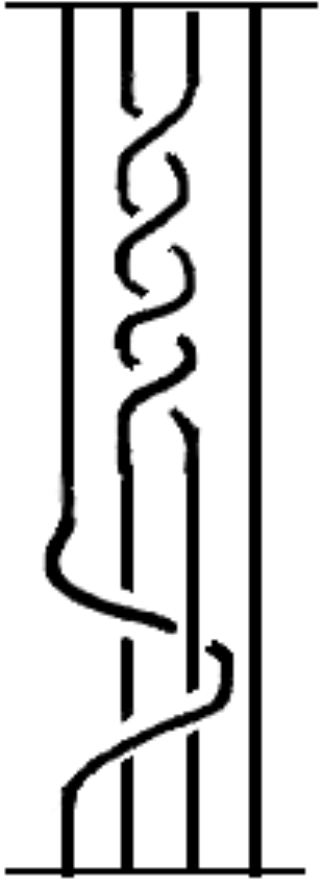}
\caption{The braid $A_2$.}
\label{Fig:SpecialBraid}
\end{figure}
\end{center}

\begin{theorem}\label{Main Theorem}
For all $n,m \in \Z$, the graph $\Gamma(n,m)$ is a Brunnian theta graph. Furthermore, suppose that for a given $(n,m) \in \Z \times \Z$, the set $S(n,m) \subset \Z\times \Z$ has the properties that if $(a,b) \in S(n,m)$ then $\Gamma(a,b)$ is isotopic to $\Gamma(n,m)$ and if $(a,b), (a',b') \in S(n,m)$ are distinct, then $a + b \neq a' + b'$. Then $S(n,m)$ has at most three distinct elements. In particular, there exist infinitely many $n \in \Z$ such that the graphs $\Gamma(n,0)$ are pairwise distinct Brunnian theta graphs.
\end{theorem}

Before proving the theorem, we establish some background.

A \defn{handlebody} is the regular neighborhood of a finite graph embedded in $S^3$ and its \defn{genus} is the genus of the boundary surface. We will be considering genus 2 handlebodies.  A disc $D$ properly embedded in a handlebody $H$ whose boundary does not bound a disc in $\boundary H$ is called an \defn{essential disc} in $H$. If $H$ has genus 2 and if $D \subset H$ is an essential non-separating disc, the space $H \setminus \inter{\eta}(D)$ is homeomorphic to $S^1 \times D^2$. A knot isotopic to the  core of that solid torus is called a \defn{constituent knot} of $H$.  If $G \subset S^3$ is a spatial theta graph and if $H = \eta(G)$, then a disc $D \subset H$ intersecting an edge $e$ of $G$ exactly once transversally and disjoint from the other edges of $G$ is called a \defn{meridian disc} for $e$. Thus, if $D$ is a meridian disc for $e$, then $H \setminus \inter{\eta}(D)$ is a regular neighborhood of $K_e$. Observe that if $G$ is a theta graph and if $e \subset e$ is an edge, then any meridian disc $D$ for $e$ is an essential disc in the handlebody $\eta(G)$, as $D$ does not separate $\eta(G)$. 

If $G$ and $G'$ are spatial theta graphs such that $G$ is isotopic to $G'$ then the isotopy can be extended to an isotopy of the handlebody $\eta(G)$ to the handlebody $\eta(G')$. Furthermore, if the isotopy takes an edge $e \subset G$ to an edge $e' \subset G'$ then the isotopy  takes any meridian disc for $e$ to a meridian disc for $e'$. On the other hand, an isotopy of $\eta(G)$ to $\eta(G')$ does not necessarily correspond to an isotopy of $G$ to $G'$. Instead, an isotopy of $\eta(G)$ to $\eta(G')$ corresponds to a sequence of isotopies and ``edge slides'' of $G$. An \defn{edge slide} of an edge $e \subset G$ of a graph involves sliding one end of $e$ across edges of $G$ (see \cite{ST}.) As in Figure \ref{Fig:EdgeSlide}, an edge slide of a theta graph may convert a theta graph into a spatial graph that is not a theta graph. Conversely, any sequence of edge slides and isotopies of a graph $G$ corresponds to an isotopy of $\eta(G)$.

\begin{center}
\begin{figure}[ht]
\includegraphics[scale = .5]{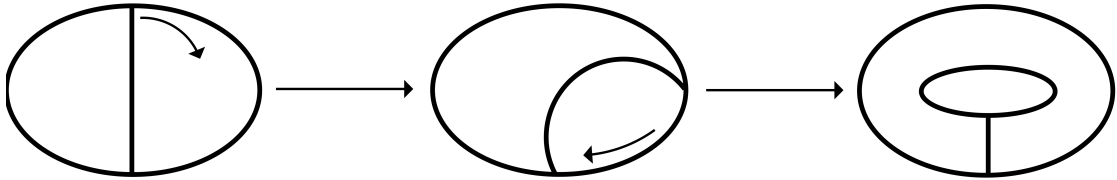}
\caption{An edge slide converting a theta graph into a non-theta graph.}
\label{Fig:EdgeSlide}
\end{figure}
\end{center}

Given a spatial theta graph $G$ and an edge $e$, an essential non-separating disc $E$ in $\eta(G)$ is \defn{along}  $e$ if it lies in a regular neighborhood of $e$, is not a meridian of $e$, if there is a meridian disc $D$ for $e$ such that $|D \cap E|$ (the number of components of $D \cap E$) is equal to $1$. Observe that if $E$ is along $e$, then $\eta(G) \setminus \inter{\eta}$ is a solid torus since $E$ is non-separating. If $E$ is along $e$, then we say that the knot which is the core of $\eta(G)\setminus \inter{\eta}(E)$ is obtained by \defn{unzipping} the edge $e$.  Figure \ref{UnzipEdge} shows two different ways of unzipping an edge. The proof of  Lemma \ref{Unzipping trivial} will also be helpful in understanding the relationship between the definition of unzipping given above and the diagrams in Figure \ref{UnzipEdge}. The term ``unzipping'' is taken from Bar-Natan and D. Thurston (see, for example, \cite{Thurston-KTG}.) It is a form of an operation also known as ``attaching a band'' to $K_e$ or ``distance 1 rational tangle replacement'' on $K_e$.

\begin{center}
\begin{figure}[ht]
\includegraphics[scale = .4]{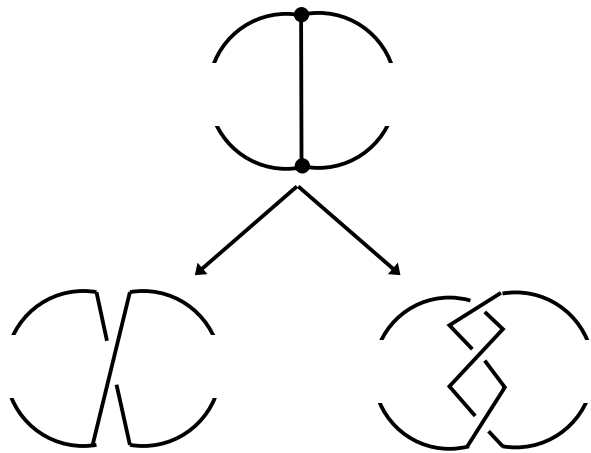}
\caption{Two ways of unzipping an edge of a spatial theta graph. As is suggested by the picture, the $\theta$-graph may be embedded in $S^3$ in some, potentially complicated, way. We do, however, require that the unzipping produce a knot and not a 2-component link.}
\label{UnzipEdge}
\end{figure}
\end{center}

Since an isotopy of a handlebody in $S^3$ to another handlebody takes discs in the first handlebody to discs in the second and preserves the number of intersections between discs, we have:

\begin{lemma}\label{isotopic unzips}
Suppose that $G$ and $G'$ are isotopic spatial theta graphs such that the isotopy takes an edge $e$ of $G$ to an edge $e'$ of $G'$. If $K \subset S^3$ is a knot obtained by unzipping the edge $e$, then there is a knot $K' \subset S^3$ which is obtained by unzipping the edge $e'$ such that $K$ and $K'$ are isotopic.
\end{lemma}

\subsection{Rational Tangles}
The key step in  our proof of Theorem \ref{Main Theorem} is to show that unzipping $\Gamma(n,m)$ does not produce any knot that can be obtained by unzipping a trivial theta graph along one of its edges. Analyzing the knots we do get will show, as a by-product, that infinitely many of the $\Gamma(n,m)$ are distinct. We use rational tangles to analyze our knots.

A \defn{rational tangle} is a pair $(B,\tau)$ where $B$ is a 3--ball and $\tau \subset B$ is a properly embedded pair of arcs which are isotopic into $\boundary B$ relative to their endpoints. We mark the points $\boundary \tau \subset \boundary B$ by NW, NE, SW, and SE as in Figure \ref{Fig:RationalTangleBasics}. Two rational tangles $(B,\tau)$ and $(B,\tau')$ are \defn{equivalent} if there is a homeomorphism of pairs $h\co (B,\tau) \to (B',\tau')$ which fixes $\boundary B$ pointwise. Conway \cite{Conway} showed how to associate a rational number $r \in \Q \cup \{1/0\}$ to each rational tangle in such a way that two rational tangles are equivalent if and only if they have the same associated rational number. We briefly explain the association, using the conventions of \cite[Lecture 4]{Gordon}. Using the 3-ball with marked points as in Figure \ref{Fig:RationalTangleBasics}, we let the rational tangle $\mc{R}(0/1)$ consist of a pair of horizontal arcs having no crossings  and we associate to it the rational number $0 = 0/1$. The rational tangle $\mc{R}(1/0)$, consisting of a pair of vertical arcs having no crossings, is given the rational number $1/0$ (thought of as a formal object.) Let $h \co B \to B$ and $v\co B \to B$ be the horizontal and vertical half-twists, as shown in Figure \ref{Fig:RationalTangleBasics}. Observe that the rational tangle $v^{2k}\mc{R}(0/1)$ is a twist box with $-k$ full twists, using the orientation convention from earlier.

\begin{center}
\begin{figure}[h!]
\labellist \small\hair 2pt
\pinlabel $NW$ [br] at 11 472
\pinlabel $NE$ [bl] at 138 472
\pinlabel $SW$ [tr] at 13 343
\pinlabel $SE$ [tl] at 144 342
\pinlabel $\mc{R}(1/0)$ [b] at 248 484
\pinlabel $\mc{R}(0/1)$ [b] at 423 484
\pinlabel $h$ [b] at 219 242
\pinlabel $v$ [b] at 219 111
\endlabellist
\includegraphics[scale = .4]{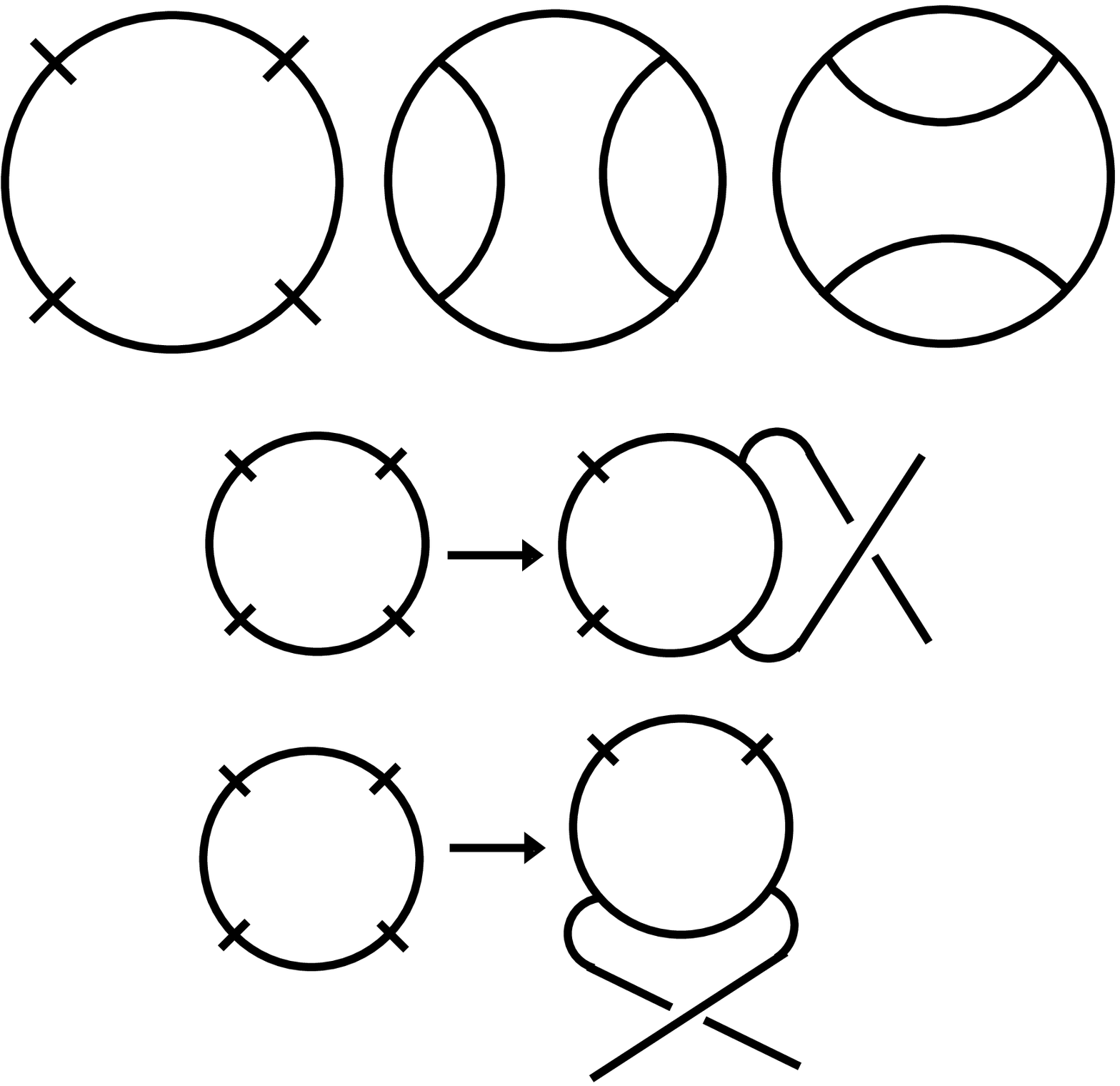}
\caption{The basic transformations of a rational tangle}
\label{Fig:RationalTangleBasics}
\end{figure}
\end{center}

Let $a_1, a_2, \hdots, a_k$ be  a finite sequence of integers such that $a_2, \hdots, a_k \neq 0$. Let $\mc{R}(a_1, \hdots, a_k)$ be the rational tangle defined by
\[
\mc{R}(a_1, a_2, \hdots, a_k) = \begin{cases}
h^{a_1}v^{a_2} \cdots h^{a_{k-1}}v^{a_k}\mc{R}(1/0) & k \text{ even }\\
h^{a_1}v^{a_2} \cdots v^{a_{k-1}}h^{a_k}\mc{R}(0/1) & k \text{ odd }\\
\end{cases}
\]
We assign the rational number
\[
p/q = a_1 + \frac{1}{a_2 + \frac{1}{a_3 + \frac{1}{\cdots + \frac{1}{a_k}}}}
\]
to $\mc{R}(a_1, a_2, \hdots, a_k)$ and we define $\mc{R}(p/q) = \mc{R}(a_1, \hdots, a_k)$, with $p$ and $q$ relatively prime.

We define the \defn{distance} between two rational tangles $\mc{R}(p/q)$ and $\mc{R}(p'/q')$ to be $\Delta(p/q,p'/q') = |pq' - p'q|$. Observe that in the 3-ball $B$, there is a disc $D \subset B$ such that $\boundary D$ partitions the marked points $\{NW, SW, NE, SE\}$ into pairs and which separates the strands of a given rational tangle $\mc{R}(p/q)$. Indeed, given a disc $D \subset B$ whose boundary partitions the marked points into pairs, there is a rational tangle $\mc{R}(p/q)$ (unique up to equivalence of rational tangles) such that $D$ separates the strands of $\mc{R}(p/q)$. We call $D$ a \defn{defining disc} for $\mc{R}(p/q)$. If $D$ is a defining disc for $\mc{R}(p/q)$ and $D'$ is a defining disc for $\mc{R}(p'/q')$ such that, out of all such discs, $D$ and $D'$ have been isotoped to intersect minimally, then it is not difficult to show that $\Delta(p/q,p'/q') = |D \cap D'|$ (i.e. the distance between the rational tangles is equal to the minimum number of arcs of intersection between defining discs.)

From a rational tangle $\mc{R}(p/q)$ we can create the unknot or a 2-bridge knot or link $\mc{K}(p/q) = \mc{D}(\mc{R}(p/q))$ by taking the so-called \defn{denominator closure} $\mc{D}$ of $\mc{R}(p/q)$ where we attach the point NW to the point SW and the point NE to the point SE by an unknotted pair of arcs lying in the exterior of $B$, as in Figure \ref{Fig:DenomClosure}. Thus, the right-handed trefoil is $\mc{K}(1/3)$ and the left-handed trefoil is $\mc{K}(-1/3)$. 

\begin{center}
\begin{figure}[h!]
\labellist \small\hair 2pt
\pinlabel $\mc{R}(p/q)$ at 191 88
\endlabellist
\includegraphics[scale = .4]{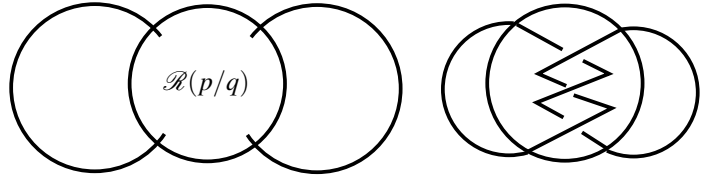}
\caption{On the left is the denominator closure $\mc{D}(\mc{R}(p/q))$ of the rational tangle $\mc{R}(p/q)$. On the right, we see that the right-handed trefoil is the denominator closure of the rational tangle $\mc{R}(1/3)$.}
\label{Fig:DenomClosure}
\end{figure}
\end{center}

\begin{theorem}[{Schubert \cite{Schubert}}]\label{Schubert}
Let $p/q, p'/q' \in \Q \cup \{1/0\}$ with $q, q' > 0$ and the pairs $p,q$ and $p',q'$ relatively prime. The knot or link $\mc{K}(p/q)$ is isotopic (as an unoriented knot or link) in $S^3$ to the knot or link $\mc{K}(p'/q')$ if and only if $q = q'$ and either $p \equiv p' \mod q$ or $pp' \equiv 1 \mod q$.
\end{theorem}

\begin{remark}
For more on Schubert's theorem, see \cite[Theorem 8.7.2]{Cromwell} or \cite[Theorem 3]{KL}. Since we are using the denominator closure of rational tangles our convention and the statement of Schubert's theorem differ from the usual convention and statement by exchanging numerators and denominators. See the discussion following Theorem 3 of \cite{KL}.
\end{remark}

\subsection{Unzipping the trivial graph}
Since we want to show that each graph in a certain family of graphs is non-trivial, the following will be useful.

\begin{lemma}\label{Unzipping trivial}
Suppose that $G \subset S^3$ is the trivial theta graph and that $K$ is a knot obtained by unzipping an edge $e$ of $G$. Then either $K$ is the unknot or there exists $k \in \Z$, odd such that $K$ is a $(2,k)$ torus knot.
\end{lemma}

\begin{proof}
Let $G$ be the trivial theta graph and let $e \subset G$ be an edge. Observe that there is an isotopy of $G$ which interchanges any two edges. Thus, we may consider $G$ to be the union of the unit circle in $\R^2$ with a horizontal diameter $e$, as in Figure \ref{Fig:RationalTangles}. We may consider the neighborhood $\eta(e)$ of $e$ as a 3-ball $B$ with a vertical disc as a meridian disc for $e$. The graph $G$ intersects $B$ in four punctures, which we label NW, NE, SW, and SE as usual. Take $D$ to be the meridian disc for $e$ and let $E \subset B$ be a disc with boundary an essential curve in $\boundary B \setminus G$, which cannot be isotoped to be disjoint from $D$, and for which $|D \cap E| = 1$. Observe that $D$ is the defining disc for the rational tangle $\mc{R}(1/0)$. If $E$ is the defining disc for the rational tangle $\mc{R}(k/\ell)$, then
\[
1 = \Delta(1/0, k/\ell) = |\ell|.
\] 
Consequently, the rational tangle $\mc{R}(k/1)$ consists of $k$ horizontal half twists. Thus the knot which is the core of $\eta(G) \setminus \inter{\eta}(E)$ is a $(2,\pm k)$ twist knot. 
\end{proof}

\begin{center}
\begin{figure}[ht]
\labellist \small\hair 2pt
\pinlabel $G$ [r] at 180 253
\pinlabel $B$ [l] at 465 253
\pinlabel $\boundary E$ [bl] at 83 66
\pinlabel $\tau_E$ [l] at 341 75
\endlabellist
\includegraphics[scale = .5]{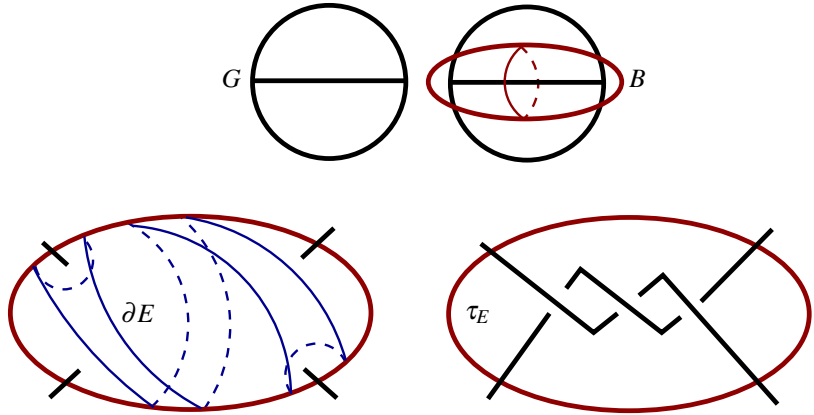}
\caption{Upper Left: The trivial graph $G$. Upper Right: the ball $B = \eta(e)$. Lower Left: A disc $E$. Lower Right: The rational tangle $\mc{R}(-3)$ with defining disc $E$.}
\label{Fig:RationalTangles}
\end{figure}
\end{center}

The following corollary follows immediately from Lemmas \ref{isotopic unzips} and \ref{Unzipping trivial}.

\begin{corollary}
Suppose that $G \subset S^3$ is a trivial spatial theta graph. Then for all edges $e \subset G$ and for any knot $K$ obtained by unzipping $e$ there exists an odd $k \in \Z$ such that $K$ is a $(2,k)$ torus knot, i.e. $\mc{K}(1/k)$.
\end{corollary}

\subsection{Proof of Theorem \ref{Main Theorem}}

The proof is similar in spirit to \cite[Section 3]{Wolcott}. We do not, however, use Wolcott's Theorem 3.11 as that theorem would require us to work with links, rather than with knots. Potentially, however, a clever use of \cite[Theorem 3.11]{Wolcott} would show that a much wider class of braids $A$ create non-trivial graphs $G(A,t_1,t_2)$. Our method, however, also allows us, using a result of Eudave-Mu\~noz concerning reducible surgeries on strongly invertible knots, to show that we have infinitely many distinct Brunnian theta graphs.

Let $n,m \in \Z$,  and let $G = \Gamma(n,m)$. To prove that $G$ is a Brunnian theta graph, by Theorem \ref{Thm: has Brunnian property}, we need only show that $G$ is non-trivial.

Let $v_+$ be the upper vertex of $G$ in Figure \ref{Fig:Template} and let $v_-$ be the lower vertex. Recall that we color the edges of $G$ (from left to right at each vertex) as blue, red, and verdant. Isotope $G$ so that the endpoint of the verdant edge $V$ adjacent to $v_-$ is moved near $v_+$ by sliding it along the red edge, as on the left of Figure \ref{Fig:UnzippingR}. This isotopy creates a diagram of $G$ such that red edge has no crossings. Let $K$ be the knot obtained by unzipping the red edge, as in the middle of Figure \ref{Fig:UnzippingR} (choosing the unzip so that no twists are inserted in the diagram along $V$). Using the doubled strands in the top braid box, isotope $K$ so that it has the diagram on the right of  Figure \ref{Fig:UnzippingR}.

Inserting the braid $A_n$ into the template, as specified in Figure \ref{Fig:Template}, our knot $K$ has the diagram on the top left of  Figure \ref{Fig:SimplifyingK1}. Let $3_1$ be the right-handed trefoil. Now perform the isotopies indicated in Figure \ref{Fig:SimplifyingK1} to see that $K$ is the connected sum of $3_1$ and the knot 
\[
\mc{K}(-\frac{3}{6(m+n)+5}) = \mc{D}(\mc{R}(0, -2(m+n)-1, -1, -1, -1)).
\]
Since torus knots are prime, $K$ is not a $(2,k)$ torus knot for any $k \in \Z$ unless $\mc{K}(-\frac{3}{6(m+n)+5})$  is the trivial knot, that is $\mc{K}(-\frac{3}{6(m+n)+5})  = \mc{K}(1)$. By Schubert's theorem, this can only happen if $6(m+n)+5 = 1$, an impossibility. Thus, each $\Gamma(n,m)$ is a Brunnian theta graph.

To prove the part about distinctness, we use a theorem of Eudave-Mu\~noz and the Montesinos trick \cite{Montesinos} (see also \cite{Gordon} for a nice explanation.) We begin by showing:

\textbf{Claim:} If $a + b \neq a' + b'$, then there is no isotopy from $\Gamma(a',b')$ to $\Gamma(a,b)$ which takes the red edge of $\Gamma(a',b')$ to the red edge of $\Gamma(a,b)$.

We prove this by contradiction. Let $B \subset S^3$ be a regular neighborhood of the red edge of $\Gamma(a,b)$ and let $W = S^3\setminus \inter{B}$ be the complementary 3--ball. Mark the points of $\Gamma(a,b) \cap \boundary B$ by NE, SE, NW, SW so that a meridian disc for the red edge of $\Gamma(a,b)$ corresponds to the rational tangle $\mc{R}(1/0)$ and the disc $E$ along which we unzip $\Gamma(a,b)$ to produce $K = 3_1 \# \mc{K}(-\frac{3}{6(a+b)+5})$ corresponds to the rational tangle $\mc{R}(0/1)$. Let $\tau = K \cap W$. 

The isotopy of $\Gamma(a,b)$ to $\Gamma(a', b')$ takes $B$ to a regular neighborhood $B'$ of the red edge of $\Gamma(a',b')$. In $B'$ there is a disc $D'$ which is along the red edge of $\Gamma(a', b')$ such that unzipping $\Gamma(a', b')$ along $D'$ produces $(3_1 \# \mc{K}(-\frac{3}{6(a'+b')+5}) )$. Reversing the isotopy, takes $D'$ to a disc $D \subset B$ which is along $e$. Let $\mc{R}(p/q)$ be the rational tangle corresponding to $D$. The knot $K' = \tau \cup \mc{R}(p/q))$ is isotopic to the result of unzipping $\Gamma(a',b')$ along $D'$ and so $K' = (3_1 \# \mc{K}(-\frac{3}{6(a'+b')+5}))$. If the disc $D$ is isotopic to the disc $E$, the rational tangles $\mc{R}(p/q)$ and $\mc{R}(0/1)$ are equivalent. In which case, $K$ is isotopic to $K'$. But this implies that $a + b = a' + b'$, a contradiction. Thus, the rational tangles $\mc{R}(0/1)$ and $\mc{R}(p/q)$ are distinct (since the discs are not isotopic). 

Since $\tau \cup \mc{R}(1/0)$ is the unknot in $S^3$, the double branched cover of $W$ over $\tau$ is the exterior of a strongly invertible knot $L \subset S^3$. Since $K = \tau \cup \mc{R}(0/1)$ and $K' = \tau \cup \mc{R}(p/q)$ are composite knots, the double branched covers of $S^3$ over $K$ and $K'$ are reducible. In particular there are distinct Dehn surgeries on $L$ producing reducible manifolds. The surgeries are distinct since $\mc{R}(0/1)$ is not equivalent to $\mc{R}(p/q)$. However this contradicts the fact that the Cabling Conjecture holds for strongly invertible knots \cite[Theorem 4]{EM}.\qed(Claim)

For a pair $(n,m) \in \Z \times \Z$, let $S(n,m) \subset \Z \times  \Z$ be a subset with the property that for all $(a,b) \in S(n,m)$, the graph $\Gamma(a,b)$ is isotopic to the graph $\Gamma(n,m)$ and which has the property that for all pairs $(a,b), (a',b') \in S(n,m)$ if $a + b = a' + b'$, then $(a,b) = (a',b')$. Observe that $(n,m) \in S(n,m)$. We will show that for all $(n,m) \in \Z \times \Z$, the set $S(n,m)$ has at most three elements. 

Suppose, for a contradiction, that there exists $(n,m)$ such that $S(n,m)$ has at least 4 distinct elements \[(a_1, b_1), (a_2, b_2), (a_3, b_3), (n,m).\] Each isotopy between any two of the graphs $\{\Gamma(a_1, b_1), \Gamma(a_2, b_2), \Gamma(a_3, b_3), \Gamma(n,m)\}$ induces a permutation of the set $\{B,R,V\}$ of blue, red, and verdant edges. For each $i \in \{1,2,3\}$, choose an isotopy $f_i$ from $\Gamma(n, m)$ to $\Gamma(a_i, b_i)$ and let $\sigma_i$ be the induced permutation of $\{B,R,V\}$. By the claim and the definition of $S(n,m)$, no $\sigma_i$ fixes $R$ and, whenever $i \neq j$, the permutation $\sigma_i \sigma_j^{-1}$ also does not fix $R$. Hence, $\sigma_i \neq \sigma_j$ if $i \neq j$. In the permutations of the set $\{B,R,V\}$, there are exactly four that do not fix $R$ and of those, two are transpositions. Thus, without loss of generality, we may assume that $\sigma_1$ is a transposition. 

Suppose that $\sigma_1$ is the transposition $(B,R,V) \to (B, V, R)$. Since neither $\sigma_2\sigma_1^{-1}$ nor $\sigma_3\sigma_1^{-1}$ fixes $R$ and since $\sigma_2 \neq \sigma_3$, the permutations $\sigma_2$ and $\sigma_3$ are the two permutations taking $R$ to $B$. But then the composition $\sigma_2\sigma_3^{-1}$ takes $R$ to $R$, a contradiction.  The case when $\sigma_1$ is the transposition $(B,R,V) \to (R,B,V)$ similarly gives rise to a contradiction. Thus, for every $(n,m) \in \Z \times \Z$, the set $S(n,m)$ has at most three elements (including $(n,m)$.)

Define a sequence $(n_i)$ in $\Z$ recursively. Let $n_1 \in \Z$ and recall that, by the above, $\Gamma(n_1,0)$ is a Brunnian theta graph.  Assume we have defined $n_1, \hdots, n_i$ so that the graphs $\Gamma(n_j,0)$ for $1 \leq j \leq i$ are pairwise non-isotopic Brunnian theta graphs. Let $P \subset \Z$ be such that $n \in P$ if and only if $\Gamma(n,0)$ is isotopic to $\Gamma(n_j,0)$ for some $1 \leq j \leq i$. Since for each $j$ with $1 \leq j \leq i$ there are at most 3 elements $n$ of $\Z$ such that $\Gamma(n,0)$ is isotopic to $\Gamma(n_j, 0)$, the set $P$ is finite. Hence, we may choose $n_{i+1} \in \Z\setminus P$. Thus, we may construct a sequence $(n_i)$ in $\Z$ so that the graphs $\Gamma(n_i,0)$ are pairwise disjoint Brunnian theta graphs. \qed

\begin{center}
\begin{figure}[ht]
\includegraphics[scale = .25]{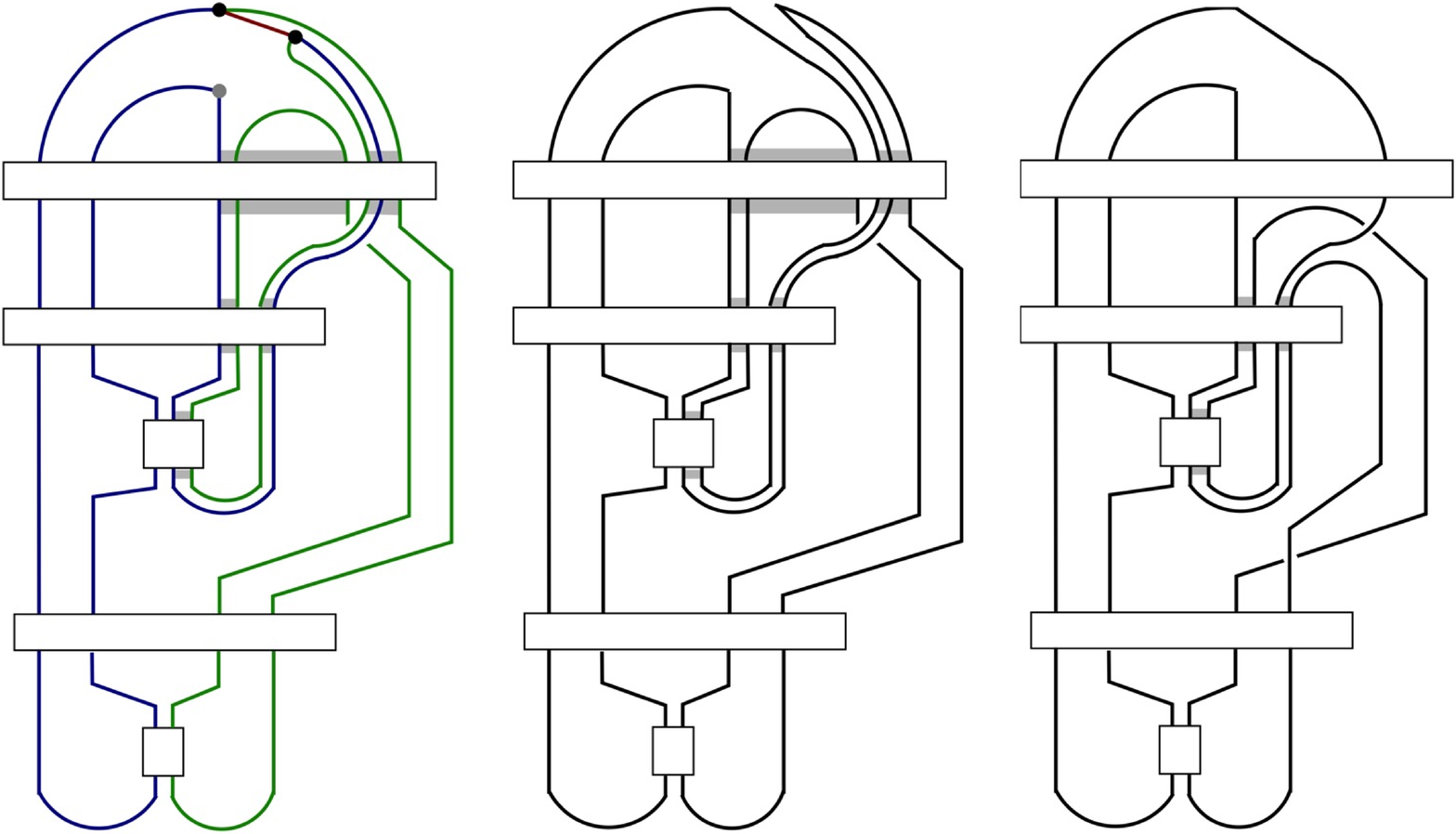}
\caption{On the left we have isotoped the template so that $R$ has no crossings. The shaded boxes denote doubled strands. In the middle we have unzipped along $R$ using a particular choice of unzipping disc. On the right, we have simplified the unzipped knot by using the parallel strands from the top braid box.}
\label{Fig:UnzippingR}
\end{figure}
\end{center}

\begin{center}
\begin{figure}[ht]
\labellist \small\hair 2pt
\pinlabel ${\Large \#}$ at 199 435
\pinlabel $-n$ at 127 1632
\pinlabel $n$ at 139 1216
\pinlabel $-n$ at 139 1164
\pinlabel $n$ at 129 909
\pinlabel $m$ at 129 876 
\endlabellist
\includegraphics[scale = .31]{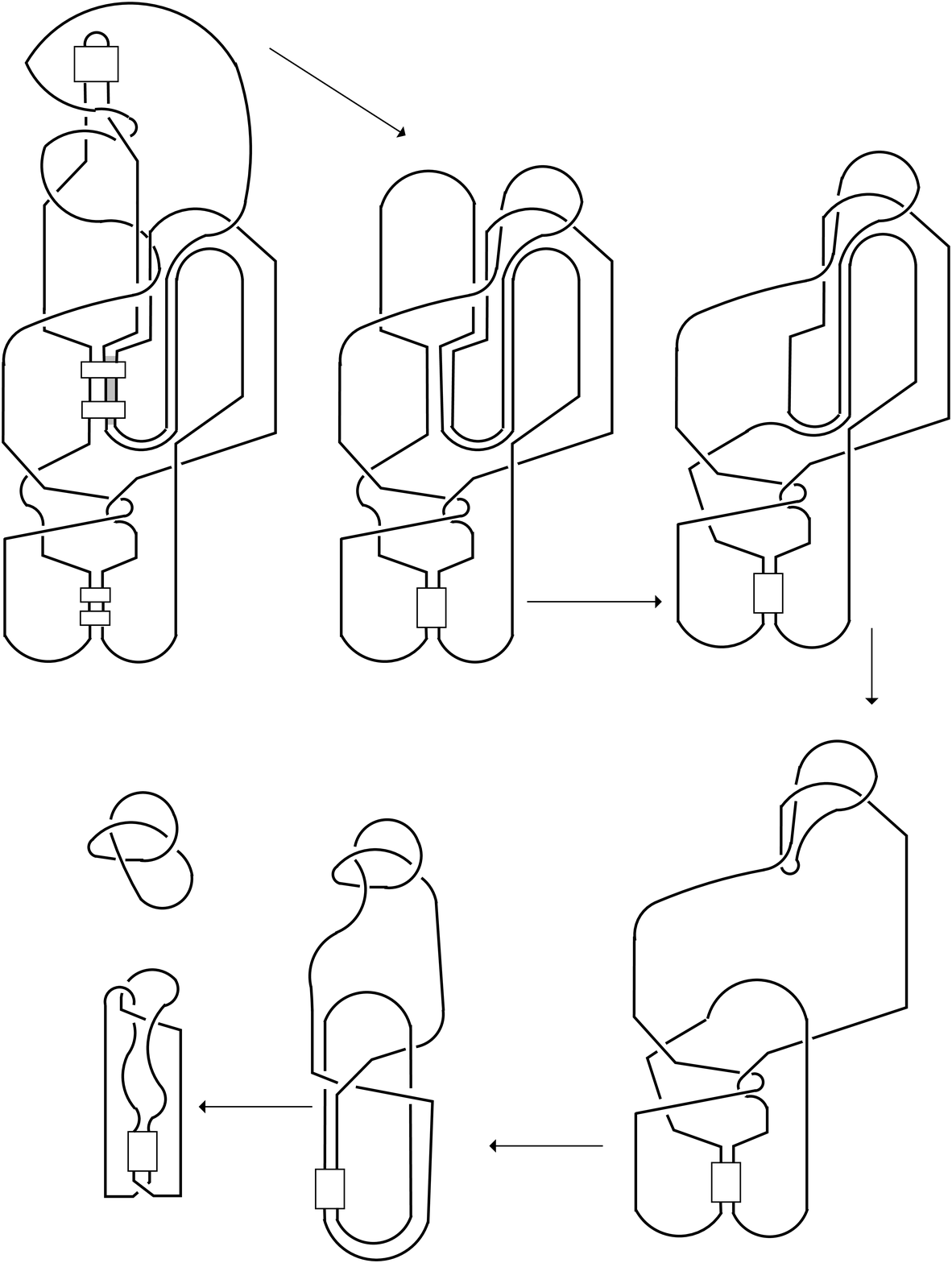}
\caption{The isotopies showing that $K$ is the connected sum of a right handed trefoil and $\mc{K}(-\frac{3}{6(m+n)+5}) $. In the first step we combine the lower two twist boxes into a single twist box with $m + n$ full twists.}
\label{Fig:SimplifyingK1}
\end{figure}
\end{center}

\section{Questions and Conjectures}\label{Questions}
Using the software \cite{Orb}, and a lot of patience, it is possible to compute (approximations to) hyperbolic volumes for some of the graphs $G(A,t_1,t_2)$. Our explorations suggest that ``most'' of the braids $A \in PB(4)$ produce hyperbolic Brunnian theta graphs for all $t_1, t_2 \in \Z$. Indeed, the software suggests that for a ``sufficiently complicated'' braid $A \in PB(4)$, and for fixed $t_1, t_2$ the volume of the exterior of $G(A^n,t_1,t_2)$ grows linearly in $n$. This is to be contrasted with the belief, based on the Thurston $2\pi$-theorem, that for a fixed $A$ and $t_1$, the volumes of $G(A,t_1, t_2)$ will converge as $t_2 \to \infty$. Furthermore, calculations of hyperbolic volumes using \cite{Orb} indicate that the graphs $\Gamma(n,m)$ of Theorem \ref{Main Theorem} are likely not Kinoshita-Wolcott graphs. Since the calculations of hyperbolic volume are only approximate and since we can only calculate the volumes of finitely many of the graphs, we do not have a proof of that fact. 

These investigations raise the the following questions.
\begin{enumerate}\spacing
\item For what braids $A \in \ker \phi$ is $G(A,0,0)$ a Brunnian theta graph?
\item Can Litherland's Alexander polynomial (or some other algebraic invariant) prove that there are infinitely many braids $A \in \ker \phi$ such that $G(A,t_1,t_2)$ is a Brunnian theta graph for some $t_1, t_2 \in \Z$?
\item Is any one of the Brunnian graphs $\Gamma(n,m)$ a Kinoshita-Wolcott graph?
\item Are there infinitely many braids $A$ such that the graph $G(A,0,0)$ is a Brunnian theta graph which is not a Kinoshita-Wolcott graph? We conjecture the answer to be ``yes''.
\item For what $A \in \ker \phi$ and $t_1,t_2 \in \Z$ is $G(A,t_1,t_2)$ a \emph{hyperbolic} Brunnian theta graph? We conjecture that whenever $G(A,t_1,t_2)$ is a Brunnian theta graph, then it is hyperbolic.
\item Is it true that if $G(A,t_1,t_2)$ is hyperbolic then $G(A^n, t_1,t_2)$ is hyperbolic for all $n \in \N$? Does the hyperbolic volume of the exterior of $G(A^n, t_1, t_2)$ grow linearly in $n$?
\end{enumerate}

% Every LaTeX document must end with \end{document}.
\end{document}